\documentclass[11pt,reqno]{amsart}
\usepackage{amsmath,amssymb,mathrsfs}
\usepackage{graphicx,cite}
\usepackage{color}
\usepackage{hyperref}
\usepackage{caption}
\usepackage{subcaption}
\usepackage{todonotes}

\newtheorem{theorem}{Theorem}[section]

\newtheorem{lemma}[theorem]{Lemma}
\newtheorem{proposition}[theorem]{Proposition}

\newcommand{\m}[1]{\ensuremath{\mathbf{#1}}}

\newcommand{\ii}{{\rm i}}
\newcommand{\x}{x}

\newcommand{\updom}{R^+}
\newcommand{\downdom}{R^-}
\newcommand{\upG}{\m{G}^{(2)}_+}
\newcommand{\downG}{\m{G}^{(2)}_-}

\newcommand{\totalE}{\hat{\m{E}}}
\newcommand{\totalH}{\hat{\m{H}}}
\newcommand{\Hs}{\m{H}^s}
\newcommand{\Es}{\m{E}^s}

\newcommand{\Hr}{\m{H}^r}
\newcommand{\Er}{\m{E}^r}

\newcommand{\Ei}{\m{E}^i}
\newcommand{\Hi}{\m{H}^i}

\newcommand{\W}{\m{W}}
\newcommand{\V}{\m{V}}
\newcommand{\U}{\m{U}}
\newcommand{\Y}{\m{Y}}

\newcommand{\green}{g}
\newcommand{\g}{g_0}
\newcommand{\hatg}{\hat{g}_0}

\newcommand{\surfaceone}{\tilde{H}^{-\frac{1}{2}}(\Gamma)}
\newcommand{\surfacetwo}{\tilde{H}^{\frac{1}{2}}(\Gamma)}

\newcommand{\surfacethree}{\tilde{H}^{-\frac{1}{2}}_{div}(\Gamma)}
\newcommand{\surfacefour}{\tilde{H}^{-\frac{1}{2}}_{curl}(\Gamma)}

\newcommand{\Div}{\text{div }}
\newcommand{\Curl}{\text{curl }}
\newcommand{\surfacecurl}{\vec{curl}_{\Gamma}}

\newcommand{\trans}[1]{#1_{\perp}}

\numberwithin{equation}{section}

\allowdisplaybreaks[4]

\title[Diffraction by an aperture in a perfectly conducting screen]{A rigorous theory on electromagnetic diffraction by a planar aperture in a perfectly conducting screen}

\author{ Ying Liang}
\address{Department of Mathematics, Purdue University, West Lafayette, Indiana 47907, USA}
\email{liang402@purdue.edu}

\author{Hai Zhang}
\address{Department of Mathematics, 
 HKUST,  Clear Water Bay, Kowloon, Hong Kong SAR, China}
   \email{ haizhang@ust.hk}
   
   \thanks{H. Zhang was partially supported by Hong Kong RGC grant GRF 16304621}

\begin{document}

\maketitle
\begin{abstract}
In this paper, we revisit the classic problem of diffraction of electromagnetic waves by an aperture in a perfectly conducting plane. We formulate the diffraction problem using a boundary integral equation that is defined on the aperture using Dyadic Green's function. This integral equation turns out to align with the one derived by Bethe using fictitious magnetic charges and currents. We then investigate the boundary integral equation using a saddle point formulation and establish the well-posedness of the boundary integral equation, including the existence and uniqueness of the solution in an appropriately defined Sobolev space. 
\end{abstract}

\section{Introduction}\label{sec:intro}

The study of diffraction of electromagnetic waves by small apertures in a perfectly conducting screen has a long history \cite{Born}. Developed by  Kirchhoff,  the first approach treated electromagnetic fields as scalar fields and expressed the scalar diffracted field in terms of the incident field in the aperture \cite{Kirchhoff}. 
Rayleigh analyzed the problem in the case of a circular aperture with normally incident harmonic plane waves in the regime of long wavelength.  The scalar theory
was later extended to a vectorial theory by Stratton and Chu \cite{Stratton39} to suit the vectorial nature of electromagnetic fields. However, the vector formulation in no way improves the situation regarding the fulfillment of the boundary conditions on the opaque screen. 
To our knowledge, the first accurate treatment of this problem was due to Bethe. In his seminar paper in 1941 \cite{bethe1944theory}, he resolved the diffraction problem by a small hole on a perfect conducting screen with zero thickness. 
By using fictitious magnetic charges and currents in the hole which satisfies the boundary conditions on the conducting screen, he found the leading order term of the solution which satisfies Maxwell's
equations and the boundary conditions. Using the solution, he derived a formula for the transmittance of light through a small hole. 
His formula was later improved by Bouwkamp to include high-order terms \cite{Bouwkamp}. Additionally, Smythe derived a general formula for the diffracted waves by any kind of aperture in \cite{Smythe47, Smythe50}. His formula was rigorously justified using the Stratton--Chu theorem in \cite{Drezet06}. In a different approach, Meixner developed a method using an infinite series expansion in terms of spheroidal functions to obtain an exact solution to this problem \cite{Meixner48}. Miles considered the variational aspects of the diffraction problem in \cite{Miles49}. We also note that a boundary integral equation formulation for the diffraction problem using Dyadic Green's function was developed by 
Levine and Schwinger in \cite{Levine50}. We refer to \cite{Abajo07, Jung14} and the references therein for recent developments in the topic. 


The focus of this paper is to revisit the classic diffraction problem by formulating a boundary integral equation defined on the aperture using Dyadic Green's function. Our main objective is to derive the boundary integral equation for the diffraction problem and establish its well-posedness. To achieve this, we introduce appropriately defined Sobolev spaces for the associated boundary integral operator and employ a saddle point formulation. This allows us to prove the well-posedness of the boundary integral equation, including the existence and uniqueness of the solution through the Fredholm alternative. Our work provides a rigorous foundation for the application of the boundary integral equation method to the classic diffraction problem. It is worth noting that rigorous theories of boundary integral equations have been widely applied to treat various boundary value problems for partial differential equations, particularly for scattering problems and inverse scattering problems where the scatterers are bounded and have a certain smoothness for the boundaries, as seen in \cite{colton2012inverse, nedelec2001acoustic, ammari2009layer}. However, these theories cannot be applied to the diffraction problem considered in this paper, as the scatterer is unbounded.

We remark that the electromagnetic diffraction problem by an unbounded screen with a bounded aperture has a close relationship to the diffraction problem by its complement in the plane -- a bounded screen located at the bounded planar aperture -- as stipulated by Babinet's principle of complementary screens \cite{Born}. While this principle has been previously discussed in the literature \cite{Born, Jackson}, the uniqueness of the solution to the two diffraction problems has not been fully addressed. Our results, therefore, provide a justification for Babinet's principle.
It is worth noting that the diffraction problem by a bounded perfect conducting screen has been previously studied using the electric field integral equation in \cite{Buffa03}. The integral equation derived in our study is similar to the one obtained therein, but we establish the well-posedness of the integral equation using a saddle point formulation, in contrast to the Helmholtz decomposition employed in \cite{Buffa03}.


\section{Setup of the diffraction problem}\label{sec:2}
We consider the diffraction problem of a perfectly conducting sheet $\{r=(r_1,r_2,r_3)\in\mathbb{R}^3: r_3=0\}$ with an aperture $\Gamma_0\subset \{r=(r_1,r_2,r_3)\in\mathbb{R}^3: r_3=0\}$ by an incoming electromagnetic wave from above. 
We denote the upper and lower half spaces by $\updom$ and $\downdom$ respectively, and $\Omega= \updom \bigcup \downdom \bigcup \Gamma$. 
Let $(\Ei,\Hi)$ be the incident electromagnetic wave given by
 $$\Ei(r)=\m{p} e^{\ii k\m{m}\cdot r},\   \Hi(r)=\m{q} e^{\ii k\m{m}\cdot r}, $$
 where
$\m{m}$ denotes the direction in which the incidental wave propagates, $\m{p}$ denotes the direction of the electric field, and $\m{q}$ denotes the direction of the magnetic field.  There holds
 $$\m{p}=\m{q}\times \m{m},\ \m{q}=\m{m}\times \m{p}.$$

The total fields $\totalE$, $\totalH$  are governed by the following Maxwell equations in
the frequency domain:
\begin{eqnarray*}\label{max}
\begin{array}{rlll}
\nabla\times  \totalE&=& \ii k \mu  \totalH\ &\text{ in }   \Omega,\\
\nabla\times  \totalH &=& -\ii k\epsilon \totalE& \text{ in }  \Omega ,\\
 \totalH &=& \Hi+\Hr+\Hs& \text{ in } \updom,\\
 \totalH \cdot \m{n} &=&0& \text{ on }  \partial \Omega,\\
 \totalE\times \m{n} &=& 0&  \text{ on }  \partial \Omega,
\end{array}
\end{eqnarray*}
where $\epsilon$ denotes the electric permittivity, 
 $\mu$ denotes the magnetic permeability, and $\m{n}$ denotes the outward unit normal. Besides,  $\Er(r)=-\m{p} \cdot (I-2\m{e}_3\m{e}_3)e^{\ii k\m{m}\cdot \overline{r}},\  \Hr(r)=\m{q} \cdot (I-2\m{e}_3\m{e}_3) e^{\ii k\m{m}\cdot \overline{r}} $ are the reflected waves in the upper half space, where $\m{e}_3 = (0,0,1)^\top$ and  $\overline{r}=(r_1,r_2,-r_3)$.  $( \Es, \Hs)$ are the scattered waves due to the presence of the aperture, which propagates in all directions. We further assume that the scattered fields are subject to the  Silver--M\"{u}ller  radiation condition 
\begin{equation*}
\lim_{|r|\to\infty} |r|\left(\sqrt{\mu}\Hs(r)\times \frac{r}{|r|}-\sqrt{\epsilon}\Es(r)\right)=0.
\end{equation*}

\begin{section}{Integral equation}\label{sec:3}
\begin{subsection}{Functional spaces and surface differential operators }

We first introduce several functional spaces on the aperture $\Gamma=\{(x_1,x_2): (x_1,x_2,0)\in\Gamma_0\}\subset \mathbb{R}^2$.
Let 
\begin{eqnarray*}
\surfaceone  &=& \{f\in H^{-\frac{1}{2}}(\mathbb{R}^2): \text{supp }f\subset \overline{\Gamma}\},\\
\surfacetwo  &=& \{f\in H^{\frac{1}{2}}(\mathbb{R}^2): \text{supp }f\subset \overline{\Gamma}\},\\
H^{\frac{1}{2}}(\Gamma)  &=&\{f|_{\Gamma}: f\in H^{\frac{1}{2}}(\mathbb{R}^2)\},
\end{eqnarray*}
where $\text{supp }f$ denotes the support for $f$, and $\overline{\Gamma}$ denotes the closure of $\Gamma$. 
The corresponding norms are defined as
\begin{eqnarray*}\label{def:onenorm}
\Vert f\Vert_{\surfaceone}^2 &=&  \frac{1}{4\pi^2} \int_{\mathbb{R}^2}  (1+|\xi|^2)^{-\frac{1}{2}} |\hat{f}(\xi)|^2d\xi, \\
\Vert f\Vert_{\surfacetwo}^2 &=&  \frac{1}{4\pi^2} \int_{\mathbb{R}^2}  (1+|\xi|^2)^{\frac{1}{2}} |\hat{f}(\xi)|^2d\xi, \\
\Vert f\Vert_{H^{\frac{1}{2}}(\Gamma)} &=& \inf\{\|f\|_{H^{\frac{1}{2}}(\mathbb{R}^2)}: f\in H^{\frac{1}{2}}(\mathbb{R}^2)\},
\end{eqnarray*}
where
 $$\hat{f}(\xi) = \int_{\mathbb{R}^2} e^{-\ii x\cdot\xi} f(x)dx.$$
  It is noted that $\surfaceone$ is the dual space of $H^{\frac{1}{2}}(\Gamma)$ and vice versa; see \cite{Mclean}. For $f\in \surfaceone$, $h\in H^{\frac{1}{2}}(\Gamma)$, we denote by
$\langle f, h \rangle $
the dual pairing between the two. 
We further introduce the product space $$H = \surfaceone\times \surfaceone,$$and
\begin{eqnarray*}
X = \surfacethree &=&\{\V\in H: \Div \V \in \surfaceone \}.   
\end{eqnarray*}
Then we define the norm
\begin{eqnarray*}
\Vert \V\Vert^2_X&:=& \Vert \V\Vert^2_{\surfaceone\times \surfaceone} +\Vert \Div \V \Vert_{\surfaceone}^2\\
&=& \frac{1}{4\pi^2} \int_{\mathbb{R}^2}  \left(1+|\xi|^2\right)^{-\frac{1}{2}} \left(|\hat{\V}(\xi)|^2+|\widehat{\Div\V}(\xi)|^2\right) d\xi. 
\end{eqnarray*}

Next, we introduce the following representations of the differential operators and their Fourier transforms. By definition, for $\V = (V_1, V_2)\in H$,
\begin{equation*}
\Div \V =  \frac{\partial V_1}{\partial x_1}+ \frac{\partial V_2}{\partial x_2}, \quad \Curl \V = \frac{\partial V_1}{\partial x_2}-\frac{\partial V_2}{\partial x_1}, 
\end{equation*}
thus the Fourier transform
\begin{eqnarray*}
\widehat{\Div \V} &=&  \ii\xi_1 \hat{ V_1}+ \ii\xi_2 \hat{ V_2} =\ii|\xi|\hat{\V}\cdot\hat{\xi},\\
\widehat{\Curl \V}&=&  \ii\xi_2 \hat{ V_1}-  \ii\xi_1 \hat{ V_2} = \ii|\xi| \hat{\V}\cdot{\trans{\hat{\xi}}},
\end{eqnarray*}
where $\xi = (\xi_1,\xi_2)\in\mathbb{R}^2$,  $\trans{\xi} =(\xi_2, -\xi_1)$,  $\hat{\xi}: = \frac{\xi}{|\xi|}$ and $\trans{\hat{\xi}}: = \frac{ \trans{\xi}}{| \trans{\xi}|}$.
Then it follows that $\hat{\V}$ can be represented by
\begin{eqnarray*}
\hat{\V} &=& (\hat{\V}\cdot \hat{\xi}) \hat{\xi}+  (\hat{\V}\cdot \trans{\hat{\xi} })\trans{ \hat{\xi}}\\
&=&\frac{\widehat{\Div \V}}{\ii |\xi|}  \hat{\xi}+\frac{\widehat{\Curl \V}}{\ii |\xi|} \trans{ \hat{\xi}}.
\end{eqnarray*}

%

\end{subsection}
\begin{subsection}{Integral representation}

   In this subsection, we introduce Dyadic Green's functions \cite{tai1994dyadic} to represent the fields. The second kind of dyadic Green’s function  $\upG$ for the upper domain satisfies:
 \begin{equation*}
 \nabla\times \nabla\times \upG(r,r'; k)-k^2\upG(r,r'; k)=I\delta(r-r')
 \end{equation*}
with the boundary condition
$$ \m{e}_3\times \nabla\times\upG(r,r'; k)=0 \text{ on } \partial \updom.$$
The Dyadic Green's function $\m{G}^{(2)}_+$ can be represented by 
   \begin{eqnarray} \label{eq-G}
  \m{G}_+^{(2)}(r,r';k)=\m{G}(r,r';k)+\begin{pmatrix}
  1&0&0\\
  0&1&0\\
  0&0&-1
  \end{pmatrix} 
  \m{G}(\overline{r},r';k),
  \end{eqnarray}
  where
\begin{eqnarray*}
\m{G}(r,r';k)=(I+\frac{1}{k^2}\nabla\nabla)\green(r,r';k),
\end{eqnarray*}
and $\green$ denotes the free space scalar Green's function
\begin{equation} \label{green}
\green(r,r';k)=\frac{e^{\ii k|r-r'|}}{4\pi|r-r'|}.
\end{equation}
For the lower half-space $\downdom$, the second kind of dyadic Green’s function can be represented as
    \begin{eqnarray*}
    \m{G}_-^{(2)}(r,r';k)=\m{G}_+^{(2)}(\overline{r},r';k),
    \end{eqnarray*}
 where  $\overline{r}=(r_1, r_2, -r_3)$. 
For convenience, we drop $k$ in the notation of the Green's function $g$ and the Dyadic Green's functions.
With the second Dyadic Green's function $\upG$, we have the following representation formula \cite{tai1994dyadic, chien2000green} 
for the scattered field in the upper-half space:
  \begin{equation}\label{eq:up}
  \Hs(r)=-\ii k\int_{\Gamma_0}\upG(r,r')\cdot [\m{e}_3\times \totalE^s(r')]ds \text{ for all } r\in\updom.
   \end{equation}
Similarly, in the lower half-space,  
 \begin{equation}\label{eq:down}
\Hs(r)=\ii k\int_{\Gamma_0}\downG(r,r')\cdot [\m{e}_3\times \totalE^s(r')]ds \text{ for all } r\in\downdom.
  \end{equation}
We introduce the notation $\W(r) =  \m{e}_3\times \totalE^s(r)|_{r\in \Gamma_0}=\m{e}_3\times \totalE(r)|_{r\in \Gamma_0}$.
With the continuity of the tangential component of the total field $\totalH$, it follows from equations \eqref{eq:up} and \eqref{eq:down} that
\begin{equation*}
-\ii k \m{e}_3\times \lim_{r\in\updom,r\to \Gamma_0} \int_{\Gamma_0}\upG(r,r')\cdot \W(r')ds + \m{e}_3\times (\Hi+ \Hr) = \ii k \m{e}_3\times \lim_{r\in\downdom,r\to \Gamma_0} \int_{\Gamma_0}\downG(r,r')\cdot\W(r')ds 
\end{equation*}
which leads to 
\begin{equation*}
2\ii k \m{e}_3\times \lim \int_{\Gamma_0}\upG(r,r')\cdot \W(r')ds = \m{e}_3\times (\Hi+ \Hr),\quad r\in \Gamma_0.
\end{equation*}
Using \eqref{eq-G}, one further obtains
\begin{equation}\label{eq:thinboard}
4\ii k(\int_{\Gamma_0} \W(r') g(r,r') ds +\frac{1}{k^2} \int_{\Gamma_0} \Div \W(r') \nabla_{ r}g(r,r') ds ) =  \m{e}_3\times (\Hi+ \Hr).
\end{equation}
Note that the above integral equation is exactly the same as the one derived by Bethe in his seminar paper \cite{bethe1944theory} using fictitious magnetic charges and currents.

We now proceed to analyze the integral equation \eqref{eq:thinboard}.  We shall view $\W(\cdot)$ as a vector field defined on $\Gamma$, and still denote it as $\W$. That is, we write $\W(\x)=(W_1(x),W_2(x))$ for  $\W(r)=(W_1(r),W_2(r),0)$, where $\x=(r_1,r_2)\in\Gamma$ for $r=(r_1,r_2,0)$. With the notation $\x= (r_1,r_2)$, $\x'= (r'_1,r'_2)$ corresponding to  $r=(r_1,r_2,0)$ and $r'=(r'_1,r'_2, 0)$ respectively,  we also introduce a function $\g(\cdot,\cdot)$ defined on $\mathbb{R}^2\times\mathbb{R}^2$ by setting
\begin{eqnarray*}
\g(\x,\x') =g(r,r').
\end{eqnarray*}
 Then we introduce the following integral operator 
\begin{equation*}
L[\W](\x) =- k^2\int_{\Gamma} \W(\x') \g(\x,\x') ds - \int_{\Gamma} \Div \W(\x') \nabla_{\Gamma, \x}\g(\x,\x') ds,
\end{equation*}
where the surface gradient operator $\nabla_{\Gamma}$ is defined as 
\begin{equation*}
\nabla_{\Gamma} f(x) = (\frac{\partial f}{\partial x_1},  \frac{\partial f}{\partial x_2}),\quad x = (x_1,x_2).
\end{equation*}
Denoting $\Y(\x)= (Y_1, Y_2)$ where $ (Y_1,Y_2,0) =-\m{e}_3\times (\Hi+ \Hr)(r)$, $r=(r_1,r_2,0)\in\Gamma_0$ and $\x=(r_1,r_2)\in\Gamma$, one can deduce the variational form of equation \eqref{eq:thinboard}:
\begin{equation}\label{eq:var}
(L[\W], \V) = (\Y, \V),\quad \forall\   \V\in X, 
\end{equation}
where 
\[
(\Y, \V) =: \langle \Y, \overline {\V} \rangle.
\]
We next analyze the well-posedness of the variational problem \eqref{eq:var}. To do so, 
we need the following lemma on the estimates of the Fourier transform of a function in $H^{-\frac{1}{2}}(\mathbb{R}^2)$.
 Note that we use $C$ to denote a generic constant. The value of the generic constant is not required and may change step by step, but its meaning will always be clear from the context.

\begin{lemma}\label{lem:linfty}
For $f\in H^{-\frac{1}{2}}(\mathbb{R}^2)$ with $\text{supp }f\subset \Gamma$, there exists a constant $C$ depending on $k$ such that for all $|\xi| \leq 2k$,
\begin{equation*}
|\hat{f}(\xi)| \leq C \Vert f\Vert_{H^{-\frac{1}{2}} (\mathbb{R}^2)}.
\end{equation*}

\end{lemma}

\begin{proof}
By definition,
$\hat{f}(\xi)= \langle e^{-\ii x\cdot\xi}, f\rangle_{(H^{1/2}(\Gamma), \Tilde{H}^{-1/2}(\Gamma))}$. Let $\chi(x)$ be a smooth cut-off function with a compact support such that $\chi(x)=1$ on $\Gamma$. Then we have
\[
\|e^{-\ii x\cdot\xi}\|_{H^{1/2}(\Gamma)} \leq \|e^{-\ii x\cdot\xi} \chi(x)\|_{H^{1/2}(\mathbb{R}^2)}
\leq C
\]
for some positive constant $C$ that depends on $k$. It follows that $|\hat{f}(\xi)| \leq C \Vert f\Vert_{H^{-\frac{1}{2}} (\mathbb{R}^2)}$. 
\end{proof}

We then present a property of the scalar function $g_0$. Recall that for $r, r'\in\mathbb{R}^3$, $g(r,r')$
satisfies
\begin{equation}\label{eqn:ggreen}
\Delta g(r,r') +k^2 g(r,r') = -\delta(r-r').
\end{equation}
Since
\begin{eqnarray*}
\g(\x,\x') =g(r,r'),
\end{eqnarray*}
it follows \eqref{eqn:ggreen} that
\begin{eqnarray*}
\g(\x,\x')  = \int_{\mathbb{R}^2} \hatg(\xi) e^{\ii(\x-\x')\cdot\xi}d\xi,
\end{eqnarray*}
with
\begin{equation*}
\hatg(\xi) = \frac{\ii}{\sqrt{k^2-|\xi|^2}}.
\end{equation*}
We also have
\begin{equation*}
\W(x) = \int_{\mathbb{R}^2} \hat{\W}(\xi) e^{\ii\x\cdot\xi}d\xi.
\end{equation*}

Let us define a bilinear form $B: X\times X \to \mathbb{C}$ as
\begin{equation*}
B(\W, \V) = (L[\W], \V).
\end{equation*}
In this section, we represent the Fourier transform $\hat{\W}$ of $\W = (W_1, W_2)$ and   $\hat{\V}$ of $\V = (V_1,V_2)$ as 
$\hat{\W} = a_1\hat{\xi}+ a_2 \trans{\hat{\xi}}$ and $  \hat{\V} = b_1\hat{\xi}+b_2 \trans{\hat{\xi}}$ respectively.
Then it follows the definition of operator $L$ that
\begin{eqnarray}
B(\W, \V)  &=& -k^2 \int_{\mathbb{R}^2} \hat\green(\xi)\hat{\W}\cdot\overline{\hat{\V}}d\xi \nonumber \\
&&- \int_{\mathbb{R}^2} \ii (\xi_1 \hat{W_1}+ \xi_2\hat{ W_2})( \ii\xi_1\hatg(\xi)\bar{\hat{V}}_1+\ii\xi_2\hatg(\xi)\bar{\hat{V}}_2)d\xi \nonumber\\
&&= \int_{\mathbb{R}^2}  \hatg(\xi)\left(\left(-k^2 + |\xi|^2 \right)a_1\bar{b}_1 - k^2 a_2 \bar{b}_2\right)d\xi.\label{eq:BWV}
\end{eqnarray}

To analyze the existence and uniqueness of the solution to the variational formulation \eqref{eq:var}, we recall the following abstract theorem.

\begin{theorem}\label{thm:fred}\cite[Theorem 5.4.5]{nedelec2001acoustic}
Let $V$ and $W$ be two Hilbert spaces. Let $a(u,v)$ be a bilinear form continuous on $V\times V$ which satisfies
\begin{equation*}
\Re [a(u,v)] \geq \alpha \Vert u\Vert^2_V -c\Vert u \Vert_{V_1}^2,\, \forall v\in V,
\end{equation*}
where $V_1$ is a Hilbert space containing $V$. Let $b(q, v)$ be a bilinear form continuous on $W\times V$ which satisfies
\begin{equation*}
\sup_{\Vert u\Vert_V = 1}\Re [b(q,u)] \geq \beta \Vert q\Vert_W -c\Vert q \Vert_{W_1},\, \forall q\in W,
\end{equation*}
where $W_1$ is a Hilbert space containing $W$.
Consider the following variational problem, 
\begin{eqnarray*}
\begin{array}{rlll}
a(u,v)+b(p,v) &=& (g_1, v)\, &\forall v\in V,\\
b(q,u) &=& (g_2, q) \, &\forall q\in W.
\end{array}
\end{eqnarray*}
where $g_1\in V^\star$, $g_2\in W^\star$.
Denote by $V_0$ the kernel of the bilinear form $b$ in $V$, i.e., 
\begin{equation*}
V_0 = \{u\in V: b(q, u) = 0,\, \forall q\in W\}.
\end{equation*}
Suppose that the injection from $V_0$ into $V_1$ is compact and that the injection from $W$ into $W_0$ is compact. Suppose that there exists an element $u_{g_2}\in V$ such that
$$b(q,u_{g_2}) = (g_2, q),\forall q\in V,$$
then the variational problem satisfies the Fredholm alternative.

\end{theorem}

Now we formulate the saddle point problem corresponding to the variational formulation \eqref{eq:var}. Introduce the surface curl operator $\surfacecurl$:
\begin{equation*}
\surfacecurl  f(x) = \frac{\partial f}{\partial x_2} \hat{x}_1- \frac{\partial f}{\partial x_1} \hat{x}_2,\quad x = (x_1,x_2). 
\end{equation*}
We first decompose the function $\W = \U+\surfacecurl p$, where $p\in X_2 =  H^{\frac{1}{2}}(\mathbb{R}^2)$ and $\text{supp } p\subset \Gamma$.
We introduce a bilinear operator $$E(p,\V) = B(\surfacecurl p,\V).$$ Then a saddle point formulation associated with \eqref{eq:var} is
\begin{eqnarray*}
\begin{array}{rlll}
B(\U, \V)+  E( p, \V)&=& (\Y , \V),\quad &\forall \V\in X,\\
E(q, \U)&=&0,\quad &\forall q\in X_2.
\end{array}
\end{eqnarray*}

To prove that the variational problem \eqref{eq:var} satisfies the Fredholm alternative, we will present several lemmas on the bilinear operators in  \eqref{eq:var}. Throughout this work, the notation $a \lesssim b$ ($a \gtrsim b$) denotes $a \leq Cb$ ($a \geq Cb$) for some generic constant $C > 0$.  Throughout this work, we will use the notation $\lesssim$ ($\gtrsim$) when we are not concerned with the specific dependence of the generic constant.
\begin{lemma}\label{cond0}
The bilinear operator $B: X\times X \to \mathbb{C}$ satisfies:
\begin{equation}\label{saddle:B}
\Re[B(\U, {\U})] \gtrsim  \Vert \U\Vert_X^2 - \Vert \U\Vert_H^2. 
\end{equation}
\end{lemma}
\begin{proof}
 Represent $\U$  as 
$\U= a_1\hat{\xi} + a_2\trans{\hat{\xi}}$, where
\begin{equation*}
a_1 = \frac{\widehat{\Div \U}}{\ii|\xi|}, a_2 =  \frac{\widehat{\Curl \U}}{\ii|\xi|}.
\end{equation*} 
We have
\begin{eqnarray*}
\Re[B(\U, {\U})]  
&=&\int_{|\xi|\geq k}  \sqrt{|\xi|^2-k^2}|a_1|^2 -
\frac{k^2}{\sqrt{|\xi|^2-k^2}} |a_2|^2 d\xi \\
&=&\int_{|\xi|\geq 2k}  \sqrt{|\xi|^2-k^2}|a_1|^2 -
\frac{k^2}{\sqrt{|\xi|^2-k^2}} |a_2|^2 d\xi \\
&&+\int_{k\leq |\xi| \leq 2k}  \sqrt{|\xi|^2-k^2}|a_1|^2 -\frac{k^2}{\sqrt{|\xi|^2-k^2}} |a_2|^2 d\xi. 
\end{eqnarray*}
Recall that
\begin{eqnarray*}
\Vert \U\Vert_H^2 &= &\int_{\mathbb{R}^2}  (1+|\xi|^2)^{-\frac{1}{2}}\left( |a_1|^2+ |a_2|^2 \right)d\xi,\\
\Vert \Div \U\Vert_{H^{-\frac{1}{2}}(\Gamma)}^2&=&\int_{\mathbb{R}^2}  (1+|\xi|^2)^{-\frac{1}{2}} |\widehat{\Div \U}(\xi)|^2d\xi\\
&=&\int_{\mathbb{R}^2}  (1+|\xi|^2)^{-\frac{1}{2}} |\xi|^2|a_1|^2d\xi. 
\end{eqnarray*}
Since for $\xi \geq 2k$,
\[
\sqrt{|\xi|^2-k^2} 
 \gtrsim  (1+|\xi|^2)^{-\frac{1}{2}} |\xi|^2, \quad
\sqrt{|\xi|^2-k^2}   \gtrsim 
(1+|\xi|^2)^{-\frac{1}{2}},
\]
We have the following estimate 
\begin{eqnarray*}
\int_{|\xi|\geq 2k} \sqrt{|\xi|^2-k^2} |a_1|^2d\xi 
 &\gtrsim & \int_{|\xi|\geq 2k}  (1+|\xi|^2)^{-\frac{1}{2}} |\xi|^2|a_1|^2d\xi = \Vert \Div \U\Vert_{H^{-\frac{1}{2}}(\Gamma)}^2 -\int_{|\xi|< 2k}  (1+|\xi|^2)^{-\frac{1}{2}} |\xi|^2|a_1|^2d\xi , \\
\int_{|\xi|\geq 2k} 
\frac{k^2}{\sqrt{|\xi|^2-k^2}} |a_2|^2 d\xi & \lesssim &
\int_{|\xi|\geq 2k} (1+|\xi|^2)^{-\frac{1}{2}} |a_2|^2d\xi \leq \Vert \U\Vert_H^2. 
\end{eqnarray*}

On the other hand, using Lemma \ref{lem:linfty}, we have 
\begin{eqnarray*}
\int_{|\xi|< 2k}  (1+|\xi|^2)^{-\frac{1}{2}} |\xi|^2|a_1|^2d\xi &\lesssim & \int_{|\xi|< 2k}  (1+|\xi|^2)^{-\frac{1}{2}} |\xi|^2\Vert \U\Vert_H^2d\xi \lesssim \Vert \U\Vert_H^2,\\
\int_{k\leq |\xi| \leq 2k}  \sqrt{|\xi|^2-k^2}|a_1|^2 d\xi & \lesssim & \int_{k\leq |\xi| \leq 2k} 2k \Vert \U\Vert_H^2 d\xi \lesssim \Vert \U\Vert_H^2,\\
\int_{k\leq |\xi| \leq 2k} \frac{k^2}{\sqrt{|\xi|^2-k^2}} |a_2|^2 d\xi & \lesssim & \int_{k\leq |\xi| \leq 2k}\frac{k^2}{\sqrt{|\xi|^2-k^2}}  \Vert \U\Vert_H^2 d\xi  \lesssim  \Vert \U\Vert_H^2.
\end{eqnarray*}

Therefore we can conclude that 
$$\Re[B(\U, {\U})]   \gtrsim   \Vert \Div \U\Vert_{\tilde{H}^{-\frac{1}{2}}(\Gamma)}^2-\Vert \U\Vert_H^2.$$
\end{proof}
The next Lemma is on the properties of the bilinear operator $E$.
\begin{lemma}\label{cond1}
The bilinear operator $E: \surfacetwo\times X\to \mathbb{C}$ satisfies
\begin{equation} \label{eq-est}
\max_{\Vert U\Vert_X = 1}\Re [E(q,\U) ] \gtrsim \Vert q\Vert_{\surfacetwo } - \Vert q\Vert_{\surfaceone}.
\end{equation}
\end{lemma}

\begin{proof}
Consider $\U =\surfacecurl q$. Since $\Div \U =0$, 
$\Vert \U\Vert_X  = \Vert \U\Vert_H \approx \Vert q \Vert_{\surfacetwo}$.  We then compute 
\begin{eqnarray*}
\Re [E(q,\U)]  =\Re[ B(\surfacecurl q, \U) ]
=\Re [B(\surfacecurl q, \surfacecurl q)]. 
\end{eqnarray*}
Noting that $\widehat{\surfacecurl q }(\xi) = \hat{q}(\xi) |\xi| \hat{\xi}_{\perp}$, we obtain
\begin{align*}
  \Re [ B(\surfacecurl q, \surfacecurl q)]  &= \int_{|\xi|>k}  \frac{k^2}{\sqrt{|\xi|^2-k^2}}  |\hat{q}|^2 |\xi|^2d\xi \\
&= \int_{|\xi|>2k}  \frac{k^2}{\sqrt{|\xi|^2-k^2}}  |\hat{q}|^2 |\xi|^2d\xi  + \int_{k\leq |\xi|<2k}  \frac{k^2}{\sqrt{|\xi|^2-k^2}}  |\hat{q}|^2 |\xi|^2d\xi.  
\end{align*}
Since
\[
 \frac{1}{\sqrt{|\xi|^2-k^2}} |\xi|^2 \gtrsim  \sqrt{ |\xi|^2 +1}, 
\] 
\[
\int_{|\xi|>2k}  \frac{k^2}{\sqrt{|\xi|^2-k^2}}  |\hat{q}|^2 |\xi|^2d\xi  \gtrsim \Vert q\Vert_{\surfacetwo} - \int_{k< |\xi|<2k}  \frac{k^2}{\sqrt{|\xi|^2-k^2}}  |\hat{q}|^2 |\xi|^2d\xi. 
\]
On the other hand, by Lemma \ref{lem:linfty}, we have
\begin{align*}
\int_{k\leq |\xi| \leq 2k} \frac{k^2}{\sqrt{|\xi|^2-k^2}}  |\hat{q}|^2 |\xi|^2d\xi & \lesssim  \int_{k\leq |\xi| \leq 2k} \frac{k^2}{\sqrt{|\xi|^2-k^2}}  \Vert q \Vert_{\surfaceone}^2 d\xi 
\lesssim  \Vert q \Vert_{\surfaceone}^2. 
\end{align*}
Therefore
\[
\Re[ E(q,\U)]  = \Re[B(\surfacecurl q, \surfacecurl q) ] \gtrsim  \Vert q\Vert_{\surfacetwo}^2 - \Vert  q\Vert_{\surfaceone}^2,
\]
whence the estimate \ref{eq-est} follows by recalling that 
$\Vert U\Vert_X  \lesssim \Vert q \Vert_{\surfacetwo}$.

\end{proof}

\begin{lemma}\label{cond3}
Let $X_0 := \{\V\in X: E(q,\V) = 0,\  \forall q\in \surfacetwo\}$. Then
$X_0$ is compactly embedded in $H$.
\end{lemma}

\begin{proof}
We only need to prove that for any $\V\in X_0$,  $\Curl \V =0$. As a consequence, $X_0 \subset \surfacethree\cap \surfacefour$ and hence is compactly embedded in $H$.

For this purpose, we consider the map
$T: C_0^{\infty}(\Gamma) \longrightarrow C^{\infty}(\mathbb{R}^2)$ defined by
\begin{equation} \label{eq-T}
T[p](x) = \int_{\Gamma} g(x, y)p(y)dy,
\end{equation}
or equivalently,
$$
\widehat{T[p]}(\xi)= \hatg(\xi) \hat{p}(\xi).
$$
It is clear that $T$ extends to a bounded map from $\tilde{H}^{-1/2}(\Gamma)$ to $H^{1/2}(\Gamma)$. Moreover, $T: \tilde{H}^{-1/2}(\Gamma) \to H^{1/2}(\Gamma)$ is invertiable with a bounded inverse (See Proposition \ref{prop-T} in the Appendix). 
Since $C_0^{\infty}(\Gamma)$ is dense in $\tilde{H}^{-1/2}(\Gamma)$, the subspace $T(C_0^{\infty}(\Gamma))$ is dense in $H^{1/2}(\Gamma)$. We note that $T(C_0^{\infty}(\Gamma)) \subset H^{3/2}(\Gamma) \subset H^{1/2}(\Gamma)$. 
Therefore $T(C_0^{\infty}(\Gamma))$ is dense in 
$H^{3/2}(\Gamma)$. 

Now for any $\phi \in T(C_0^{\infty}(\Gamma)) \subset H^{3/2}(\Gamma)$, we claim that 
$$
\langle{\overline{\Curl \V}}, {\phi}\rangle_{(\tilde{H}^{-3/2}, H^{3/2})} = \int_{\mathbb{R}^2}   {\overline{\widehat{\Curl \V}(\xi)}}{\hat{\phi}(\xi)}d\xi =0, 
$$
where the bracket $\langle \cdot, \cdot\rangle$ denote the dual pair between $\tilde{H}^{-3/2}(\Gamma)$ and $H^{3/2}(\Gamma)$. 
Then as a consequence, $\langle{\overline{\Curl \V}}, {\phi}\rangle_{(\tilde{H}^{-3/2}, H^{3/2})}=0$ for any $\phi \in H^{3/2}(\Gamma)$ and therefore 
$\overline{\Curl \V}=0$. It follows that $\Curl \V=0$. To prove the claim, we write $\phi = T {p}$ for some $p \in C_0^{\infty}(\Gamma)$.  Note that
\[
 B(\surfacecurl {p}, \V)=0. 
\]
Using  \eqref{eq:BWV} and the decomposition
\begin{equation*}
\hat{\V} =\frac{\widehat{\Div \V}}{\ii|\xi|}  \hat{\xi} +\frac{\widehat{\Curl \V}}{\ii|\xi|} \trans{\hat{\xi}},
\end{equation*}
we further obtain
\begin{eqnarray*}
0 
& = &- \ii\int_{\mathbb{R}^2} k^2 \hatg(\xi) |\xi| \hat{{p}}(\xi)\overline{\frac{\widehat{\Curl \V(\xi)}}{i|\xi|}} d\xi\\
&=& k^2 \int_{\mathbb{R}^2}  \hatg(\xi) \hat{{p}}(\xi) \overline{\widehat{\Curl \V}}d\xi \\
&=&k^2 \int_{\mathbb{R}^2}  \hat{\phi}(\xi) \overline{\widehat{\Curl \V}}d\xi.
\end{eqnarray*}
This proves the claim and also concludes the proof of the lemma. 

\end{proof}

\begin{theorem}\label{thm:fredholm}
With Lemmas \ref{cond0}, \ref{cond1} and \ref{cond3}, the variational problem \eqref{eq:var}  satisfies the Fredholm alternative.
\end{theorem}

The following proposition claims that the variational formulation \eqref{eq:var} has a unique solution. 
\begin{proposition}\label{prop:unique}
If $\W$ satisfies
$B(\W, \W) =0$, then $\W=0$.
\end{proposition}
\begin{proof}
Denoting $\hat{\W }= a_1\xi + a_2 \trans{\xi}$, it follows $B(\W,\W) = 0$ that 
\begin{eqnarray*}
\Im[B(\W, {\W})]  &=&\Im[ \int_{\mathbb{R}^2}  \frac{\ii}{\sqrt{k^2-|\xi|^2}}\left(\left(-k^2 + |\xi|^2 \right)a_1\bar{a}_1 - k^2 a_2 \bar{a}_2\right)d\xi]\\
&=&-\int_{|\xi|<k}  \left(\sqrt{k^2 - |\xi|^2}|a_1|^2+ \frac{k^2}{\sqrt{k^2 - |\xi|^2}} |a_2|^2\right)d\xi\\
&=&0.
\end{eqnarray*}
which leads to $a_1(\xi) = a_2(\xi) = 0$, that is,  $\hat{\W}(\xi)=0$ for all $|\xi|<k$. Since $\text{supp }\W\subset \Gamma$, $\hat{\W}$ is analytic in $\xi$, we obtain that $\W(\xi) = 0$ for $\xi \in \mathbb{C}$. Then $\W=0$.
\end{proof}

With Theorem \ref{thm:fredholm} and Proposition \ref{prop:unique}, we finally establish the following main result of the paper.  

\begin{theorem}
Let $X^*$ denote the dual space of $X$. For any $\Y \in X^*$,  there exists one unique solution $\W \in X$ to the boundary integral equation \eqref{eq:var}. Moreover, 
\[
\|\W\|_X \lesssim \|\Y\|_{X^*}.
\]

\end{theorem}

\end{subsection}

\end{section}

\begin{appendix}
\begin{section}{Scalar diffraction theory for a planar aperture in a perfect reflecting screen}
We consider the diffraction problem of a perfectly reflecting sheet $\{r=(r_1,r_2,r_3)\in\mathbb{R}^3: r_3=0\}$  with an aperture $\Gamma_0\subset \{r=(r_1,r_2,r_3)\in\mathbb{R}^3: r_3=0\}$ by an incoming scalar wave from above. We use the same notation as in Section 2. Let $u^i(r)=e^{\ii k\m{m}\cdot r}$ be the incident wave where $\m{m}=(m_1, m_2, m_3)$ is the incident direction, $k$ the wave number. We impose the Neumann boundary on the screen. The reflected wave in the absence of aperture is denoted as $u^r(r)=e^{\ii k\m{m}'\cdot r}$, where ${m}'=(m_1, m_2, -m_3) $. The total field consists of the incident field, the reflected field, and the scattered field which is denoted by $u^s$ in the upper-half space and consists of only the scattered field in the lower-half space. The scattered field satisfies the so-called Sommerfeld radiation condition. In sum, 
the diffraction problem can be modeled by the following equations: 
\begin{eqnarray*}
\begin{array}{rlll}
\nabla^{2}  u +k^2 u &=& 0 \quad  &\text{ in } \mathbb{R}^3\backslash \Gamma_0,\\
u &=& u^i+u^r+u^s &\text{ in } \updom,\\
u &=& u^s &\text{ in } \downdom,\\
\frac{\partial u}{\partial r_3} &=&0 
 &\text{ on }  \partial \Omega.
 \end{array}
\end{eqnarray*}

We shall use a boundary integral equation to reformulate the diffraction problem. For this purpose, we introduce the following Green's function for the upper half space with the Neumann boundary condition:
\[
\tilde{g}(r, r') = \green(r,r';k) + \green(r,\overline{r'};k)
\]
where $\green$ is the Green's function in the free space that is defined in (\ref{green}) and    $\overline{r'} =(r_1', r_2', -r_3') $. 
Noting that $\frac{\partial u^s}{\partial r_3'} (r')=0$ for $r' \in \mathbb{R}^2 \backslash \Gamma_0$,
the scattered field in the upper half space can be represented by the following formula
\begin{equation} \label{rep1}
u^s(r) = - \int_{\Gamma_0}  \tilde{g}( r', r)  \frac{\partial u^s}{\partial r_3'}(r')dr', \quad r \in \updom
\end{equation}
By taking the trace of the above formula and making use of the continuity of the single-layer potential, the representation formula (\ref{rep1}) still holds for $r \in \Gamma_0$. 

Similarly, we define the Green's function $\tilde{g}$ for the lower-half space. We can show that the following representation formula holds
\begin{equation*} \label{rep2}
u^s(r) =  \int_{\Gamma_0}  \tilde{g}( r', r)  \frac{\partial u^s}{\partial r_3'}(r')dr', \quad r \in \downdom.
\end{equation*}
Using the continuity of the total field across the aperture $\Gamma_0$, we obtain the following boundary integral equation:
\begin{equation} \label{integral-eq}
\int_{\Gamma_0}  \tilde{g}( r', r)  \frac{\partial u^s}{\partial r_3'}(r')dr' = u^i(r), \quad r \in \Gamma_0.
\end{equation}
It is straightforward to check that $u$ is a solution to the diffraction problem if and only if $\frac{\partial u^s}{\partial r_3'}$ solves the integral equation (\ref{integral-eq})

We now consider the integral equation (\ref{integral-eq}). 
For $r=(r_1,r_2,0)$ , $r'=(r'_1,r'_2,0) \in \Gamma_0$, we introduce  $\x= (r_1, r_2)$,  $\x'=(r'_1,r'_2) \in \Gamma\subset \mathbb{R}^2$. Then with a similar argument as for the vectorial case, the functions $\frac{\partial u^s}{\partial r_3'}$ and  $u^i$ can also be viewed as functions on $\Gamma$. Then we can check that (\ref{integral-eq}) can be reduced to the following form
\[
T \left[\frac{\partial u^s}{\partial r_3'}\right] (\x)= u^i(\x), \quad \x\in \Gamma, 
\]
where the operator $T$ is defined as in (\ref{eq-T}). 

It is straightforward to see that the following conclusion holds for $T$. 

\begin{lemma}
$T$ is a bounded linear map from $\tilde{H}^{-1/2}(\Gamma)$ to   ${H}^{1/2}(\Gamma)$.   
\end{lemma}

Moreover, we can show that

\begin{lemma} \label{lem-T}
 $T: \tilde{H}^{-1/2}(\Gamma) \to {H}^{1/2}(\Gamma)$ is a Fredholm operator with index zero.    
\end{lemma}
\begin{proof}
 We write $T= T_0 + R$ where 
$T_0$ is defined by setting $k=0$ in the formula (\ref{eq-T}), and $R=T-T_0$. From the expansion of the function $g$ in wavenumber $k$, we see that $R$ is a compact operator from   $\tilde{H}^{-1/2}(\Gamma)$ to $ {H}^{1/2}(\Gamma)$. It is also straightforward to check that 
$T_0$ is bounded from   $\tilde{H}^{-1/2}(\Gamma)$ to $ {H}^{1/2}(\Gamma)$. 

We now show that $T_0$ is invertible with a bounded inverse. 
We first show that $\text{Ker } (T_0)$ is trivial. Indeed, for any $\psi \in \tilde{H}^{-1/2}(\Gamma)$, we have 
\[
T_0 [\psi](\x) = \int_{\Gamma}\frac{1}{4\pi|\x-\x'|}\psi(\x')dx'. 
\]
Therefore
\[
\widehat{T_0 [\psi]}(\xi) =  \frac{1}{|\xi| }\hat{\psi}(\xi). 
\]
It follows that 
\[
(T_0 [\psi], \psi) =
\int_{\mathbb{R}^2} \frac{1}{|\xi| }|\hat{\psi}(\xi)|^2 \geq  \int_{\mathbb{R}^2} \frac{1}{|\xi|+1 }|\hat{\psi}(\xi)|^2  = \|\psi\|_{\tilde{H}^{-1/2}(\Gamma)},
\]
whence $\text{Ker } (T_0)$ is trivial. Moreover, the above inequality also implies that $Ran(T_0)$ is closed in $ {H}^{1/2}(\Gamma)$. 
Next, we show that  the range of $T_0$, $\text{Ran } (T_0)$,  is dense in ${H}^{1/2}(\Gamma)$ and hence $T_0$ is surjective. 
By considering the adjoint operator $T^*: \tilde{H}^{-1/2}(\Gamma) \to {H}^{1/2}(\Gamma)$, we see that $T^*= T$ and hence $\text{Ker }  T^*$ is trivial. It follows that $\text{Ran }(T_0)$ is dense in ${H}^{1/2}(\Gamma)$. Therefore, we have proved that $T_0$ is bijective. By the open mapping theorem, $T_0$ is invertible with a bounded inverse.  Recall that $R$ is compact. We can conclude that $T=T_0+R$ is a Fredholm operator with an index zero. 

\end{proof}

\begin{proposition} \label{prop-T}
$T: \tilde{H}^{-1/2}(\Gamma) \to {H}^{1/2}(\Gamma)$ is invertible with a bounded inverse. 
\end{proposition}
\begin{proof}
By the result in Lemma \ref{lem-T}, we need only to show that $\text{Ker } T$ is trivial. Indeed, assume that $T\psi=0$ for some 
$\psi \in \tilde{H}^{-1/2}(\Gamma)$. We have 
\[
(T [\psi], \psi) = \int_{\mathbb{R}^2} \frac{\ii}{\sqrt{k^2 -|\xi|^2 } }|\hat{\psi}(\xi)|^2 d\xi =0.
\]
Taking the imaginary part of the above equation, we see that
$\hat{\psi}(\xi)=0$ for $|\xi|\leq k$. Similarly, taking the real part, we see that $\hat{\psi}(\xi)=0$ for $|\xi|> k$. Therefore, we can conclude that $\psi=0$. This proves that $\text{Ker }  T$ is trivial and completes the proof of the proposition. 
\end{proof}

\end{section}

\end{appendix}
\bibliographystyle{siamplain}

\end{document}